 \documentclass[draft]{article}

\usepackage{amsmath,amsfonts,amsthm,amssymb,amscd,cancel,color}
\usepackage{enumitem}
\usepackage{verbatim}
\usepackage[dvipdfmx]{graphicx}
\usepackage{ulem,color}

\renewcommand{\Re}{\mathop{\rm Re}\nolimits}

\theoremstyle{plain}
\newtheorem{theorem}{Theorem}[section]
\newtheorem{lemma}[theorem]{Lemma}
\newtheorem{proposition}[theorem]{Proposition}
\newtheorem{corollary}[theorem]{Corollary}
\theoremstyle{definition}

\theoremstyle{remark}
\newtheorem{remark}[theorem]{Remark}

\newtheorem{assumption}[theorem]{Assumption}

\newcommand{\R}{{\mathbb R}}

\newcommand{\Z}{{\mathbb Z}}

\newcommand{\N}{{\mathbb N}}

\def\im{{\rm i}}

\newcommand{\C}{\mathbb{C}}

\def\({\left(}
\def\){\right)}
\def\<{\left\langle}
\def\>{\right\rangle}

\numberwithin{equation}{section}

\begin{document}

\title{A note on small data soliton selection for nonlinear Schr\"odinger equations with potential}

\author{Scipio Cuccagna, Masaya Maeda}
\maketitle

\begin{abstract}
In this note, we give an alternative proof of the theorem on soliton selection for small energy solutions of nonlinear Schr\"odinger equations (NLS) studied in \cite{CM15APDE,CM20}.
As in \cite{CM20}  we use the notion of  Refined Profile but unlike in  \cite{CM20} we do not modify the modulation coordinates and we do not search for Darboux coordinates.
\end{abstract}

\section{Introduction}

In this note we give an alternative and, in some respect, simplified proof of the selection of small energy standing waves
for the  nonlinear Schr\"odinger equation (NLS)
\begin{equation}\label{nls}
\im \partial_t u = Hu + g(|u|^2)u,\quad (t,x)\in \R^{1+3},
\end{equation}
where $H:=-\Delta +V$ is  a  Schr\"odinger operator with $V\in \mathcal S(\R^3,\R)$ (Schwartz function) and $g \in C^\infty (\R,\R)$ satisfies $g(0)=0$ and the growth condition:
\begin{align}\label{eq:ggrowth}
\forall n\in\N\cup\{0\},\ \exists C_n>0,\ |g^{(n)}(s)|\leq  C_n \<s\>^{2-n}\text{ where } \<s\>:=(1+|s|^2)^{1/2}.
\end{align}
We consider the Cauchy problem of NLS \eqref{nls} with the initial condition $u(0)=u_0 \in H^1(\R^3,\C)$.
It is well known that the NLS \eqref{nls} is locally well-posed (LWP) in $H^1:=H^1(\R ^3, \C )$, see e.g.\  \cite{CazSemi,LPBook
}. It is easy also to conclude, by mass and energy conservation, that  for small initial data $u_0\in H^1$ the corresponding solution is globally defined.

The aim of this paper is to revisit the study of asymptotic behavior of small (in $H^1$) solutions when the Schr\"odinger operator $H$ have several simple eigenvalues.
In such situation, it have been proved that solutions decouple into a soliton and dispersive wave \cite{SW04RMP,TY02ATMP,CM15APDE}. More recently
in \cite{CM20} we have introduced the notion of Refined Profile, which simplifies significantly the proof of the result in \cite{CM15APDE}.
In this note we exploit the notion of Refined Profile of   \cite{CM20}, but we give an alternative proof of the result in  \cite{CM20} which does not exploit directly the hamiltonian structure of the NLS. In this sense, in this paper we are closer in spirit to  Soffer and Weinstein \cite{SW04RMP}  and Tsai and Yau \cite{TY02ATMP}, but our proof is at the same time simpler and with stronger results.

To state our main result precisely, we introduce some notation  and several assumptions.
The following two assumptions for the Schr\"odinger operator $H$ hold
  for generic $V$.

\begin{assumption}\label{ass:nonres}
$0$ is neither an eigenvalue nor a resonance of $H$.
\end{assumption}

\begin{assumption}\label{ass:linearInd}
There exists $N\geq 2$ s.t.\
$$\sigma_d(H)=\{\omega_j\ |\ j=1,\cdots,N\},\text{ with }\omega_1<\cdots<\omega_N<0,$$
where $\sigma_d(H)$ is the set of discrete spectrum of $H$.
Moreover, we assume all $\omega_j$ are simple and
\begin{align}\label{eq:linind}
\forall \mathbf{m} \in \Z^N\setminus\{0\},\ \mathbf{m} \cdot  \boldsymbol{\omega} \neq 0,
\end{align}
where $\boldsymbol{\omega}:=(\omega_1,\cdots, \omega_N)$.
We set $\phi_j$ to be the eigenfunction of $H$ associated to the eigenvalue $\omega_j$ satisfying $\|\phi_j\|_{L^2}=1$.
We also set $\boldsymbol{\phi}=(\phi_1,\cdots,\phi_N)$.
\end{assumption}

\begin{remark}
	The cases $N=0,1$ are easier and are not treated it in this paper.
	Unfortunately, Assumption \eqref{ass:linearInd} excludes      radial potentials $V(r)$, for $r=|x|$, where in general we should expect
eigenvalues with   multiplicity higher than one.

\end{remark}

As it is well known, $\phi_j$'s are smooth and decays exponentially.
For $s\geq 0, \gamma\geq 0$, we set
\begin{align*}
H^s_\gamma:=\{u\in H^s\ |\ \|u\|_{H^s_\gamma}:=\|\cosh(\gamma x)u\|_{H^s}<\infty\}.
\end{align*}
The following is well known.
\begin{proposition}\label{prop:gam}
There exists $\gamma_0>0$ s.t.\ for all $1\leq j\leq N$, we have $\phi_j\in \cap_{s\geq 0}H^s_{\gamma_0}$.
\end{proposition}


Using $\gamma_0>0$, we set
\begin{align*}
\Sigma^s:= H^s_{\gamma_0} \ \text{if }s\geq 0,\  \Sigma^{s}:=(H^{-s}_{\gamma_0})^* \ \text{ if } s<0, \ \Sigma^{0-}:=(\Sigma^0)^*\text{ and }\Sigma^\infty:=\cap_{s\geq 0}\Sigma^s.
\end{align*}
We will not consider any topology in $\Sigma^\infty$ and we will only consider it as a set.

In order to introduce the   notion of  refined profile, we need the following
  combinatorial set up, which is exactly the same as in \cite{CM20}.

\noindent We start the following standard basis of $\R ^N$, which we view as ``non--resonant" indices:
\begin{align}\label{eq:defnr0}
\mathbf{NR}_0:=\{\mathbf e_j \ |\ j=1,\cdots,N\},\ \mathbf e_j:=(\delta_{1j},\cdots, \delta_{Nj})\in \Z^N, \text{ $\delta_{ij}$   the Kronecker delta.}
\end{align}
More generally, the
  sets of resonant and non--resonant indices $\mathbf{R}$, $\mathbf{NR}$, are
\begin{align}\label{eq:defrnr}
\mathbf{R}:=\{\mathbf{m}\in \Z^N\ |\ \sum\mathbf{m}=1,\ \boldsymbol{\omega} \cdot \mathbf{m}>0\},\quad
\mathbf{NR}:=\{\mathbf{m}\in \Z^N\ |\ \sum\mathbf{m}=1,\ \boldsymbol{\omega} \cdot \mathbf{m}<0\},
\end{align}
where $\sum \mathbf{m}:=\sum_{j=1}^N m_j$ for $\mathbf{m}=(m_1,\cdots,m_N) \in \Z^N$.

\noindent From Assumption \ref{ass:linearInd}  it is clear that $\{\mathbf{m}\in \Z^N\ |\ \sum \mathbf{m}=1\}=\mathbf{R}\cup \mathbf{NR}$ and $\mathbf{NR}_0\subset \mathbf{NR}$.
For $\mathbf{m} =(m_1,\cdots, m_N)\in \Z^N$, we define
\begin{align}\label{eq:absmdef}
|\mathbf{m}|:=(|m_1|,\cdots,|m_N|)\in \Z^N,\ \|\mathbf{m}\|:=\sum |\mathbf{m}|=\sum_{j=1}^N |m_j|,
\end{align}
and introduce partial orders $\preceq$ and $\prec$ by
\begin{align}
\mathbf{m}\preceq \mathbf{n}\ \Leftrightarrow_{\mathrm{def}} \forall j\in \{1,\cdots,N\},\ m_j\leq n_j, \label{eq:defpreceq}\quad \text{and}\quad
\mathbf{m}\prec \mathbf{n} \ \Leftrightarrow_{\mathrm{def}} \mathbf{m}\preceq \mathbf{n}\ \text{ and } \mathbf{m}\neq \mathbf{n},
\end{align}
where $\mathbf{n}=(n_1,\cdots,n_N)$.
We define the minimal resonant indices by
\begin{align}\label{eq:defRmin}
\mathbf{R}_{\mathrm{min}}:=\{\mathbf{m} \in \mathbf{R}\ |\ \not\exists \mathbf{n}\in \mathbf{R}\ \mathrm{s.t.}\ |\mathbf{n}|\prec |\mathbf{m}|\}.
\end{align}
We also consider $\mathbf{NR}_1$  formed by the nonresonant indices not larger than   resonant indices:
\begin{align}\label{eq:defnr1}
\mathbf{NR}_1:=\{\mathbf{m}\in \mathbf{NR}\ |\ \forall \mathbf{n}\in \mathbf{R}_{\mathrm{min}},\ |\mathbf{n}| \not \prec |\mathbf{m}|\}.
\end{align}
Both $\mathbf{R}_{\mathrm{min}}$ and $\mathbf{NR}_1$ are finite sets, see \cite{CM20}  for the elementary proof.

We now introduce the  functions $\{G_\mathbf{m}\}_{\mathbf{m} \in \mathbf{R}_{\mathrm{min} }}\subset \Sigma^\infty$ which are crucial  in our analysis.
For $\mathbf{m}\in \mathbf{NR}_1$, we inductively define $\widetilde{\phi}_{\mathbf{m}}(0)$ and $g_{\mathbf{m}}(0)$ by
\begin{align}\label{eq:indefphigroot}
\widetilde{\phi}_{\mathbf{e}_j}(0):=\phi_j,\ g_{\mathbf{e}_j}(0)=0,\ j=1,\cdots,N,
\end{align}
and, for $\mathbf{m}\in \mathbf{NR}_1\setminus \mathbf{NR}_0$,  by
\begin{align}
\widetilde{\phi}_{\mathbf{m}}(0)&:=-(H-\mathbf{m} \cdot \boldsymbol{\omega})^{-1} g_{\mathbf{m}}(0),\label{eq:indefphi}\\
g_{\mathbf{m}}(0)&:=\sum_{m=1}^\infty \frac{1}{m!}g^{(m)}(0)\sum_{(\mathbf{m}_1,\cdots,\mathbf{m}_{2m+1})\in A(m,\mathbf{m})}\widetilde\phi_{\mathbf{m}_1}(0)\cdots \widetilde{\phi}_{\mathbf{m}_{2m+1}}(0),\label{eq:indefg}
\end{align}
where
\begin{align}\label{eq:defAmm}
A(m,\mathbf{m}):=\left\{ \{\mathbf{m}_j\}_{j=1}^{2m+1}\in (\mathbf{NR}_1)^{2m+1}\ |\ \sum_{j=0}^m \mathbf{m}_{2j+1}-\sum_{j=1}^m \mathbf{m}_{2j}= \mathbf{m},\ \sum_{j=0}^{2m+1} |\mathbf{m}_{j}|=|\mathbf{m}|\right\} 
\end{align}

\begin{remark}
	For each $m\geq 1$ and $\mathbf{m} \in \mathbf{NR}_1$, $A(m,\mathbf{m})$ is a finite set.
	Furthermore, for sufficiently large $m$, we have $A(m,\mathbf{m})=\emptyset$.
	Thus, even though we are expressing $g_\mathbf{m}(0)$ in \eqref{eq:indefg} by a series, the sum is finite.
\end{remark}

For $\mathbf{m}\in \mathbf{R}_{\mathrm{min}}$, we define $G_{\mathbf{m}}$ by
\begin{align}\label{eq:defG}
G_{\mathbf{m}}:=\sum_{m=1}^\infty \frac{1}{m!}g^{(m)}(0)\sum_{(\mathbf{m}_1,\cdots,\mathbf{m}_{2m+1})\in A(m,\mathbf{m})}\widetilde\phi_{\mathbf{m}_1}(0)\cdots \widetilde{\phi}_{\mathbf{m}_{2m+1}}(0).
\end{align}
\begin{remark}
	$g_{\mathbf{m}}(0)$ and $G_\mathbf{m}$ are defined similarly. We are using a different notation to emphasize that $g_{\mathbf{m}}(0)$ has
	$\mathbf{m}\in \mathbf{NR}_1$, while $G_{\mathbf{m}}$ has  $\mathbf{m}\in \mathbf{R}_{\mathrm{min}} $.
\end{remark}

The following is the nonlinear Fermi Golden Rule (FGR) assumption essential in our analysis.
\begin{assumption}\label{ass:FGR}
	For all $\mathbf{m} \in \mathbf{R}_{\mathrm{min}}$, we assume
	\begin{align}\label{eq:FGR}
	\int_{|k|^2=\mathbf{m}\cdot \boldsymbol{\omega}}|  \widehat{G}_{\mathbf{m}}(k)|^2\,dS \neq 0,
	\end{align}
	where $\widehat{G}_{\mathbf{m}}$ is the distorted Fourier transform associated to $H$.
\end{assumption}

\begin{remark}  In the case $N=2$ and $\omega_1+2(\omega_2-\omega_1)>0$, we have $G_{\mathbf{m}}=g'(0)\phi_1\phi_2^2$, which corresponds to the condition in Tsai and Yau \cite{TY1}, based on the explicit formulas in Buslaev and Perelman  \cite{BP2} and Soffer and Weinstein  \cite{SW3}.
These works are related to  Sigal  \cite{Sigal93CMP}.
 More general situations are considered in
   \cite{CM15APDE}, where however
 the
 $G_{\mathbf{m}}$ are obtained  after a certain number of coordinate changes, so that  the relation of the  $G_{\mathbf{m}}$  and the  $\phi_j$'s is not discussed in   \cite{CM15APDE} and is not easy to track.
\end{remark}

In \cite{CM20} it is proved that for  a generic nonlinear function $g$ the condition \eqref{eq:FGR} is a consequence of the following simpler one, which is similar to (11.6) in Sigal \cite{Sigal93CMP},
\begin{align}\label{eq:FGRsimplified}
	\int_{|k|^2=\mathbf{m}\cdot \boldsymbol{\omega}}| \widehat{\phi ^ \mathbf{m}} (k)|^2\,dS \neq 0 \text{ for all $\mathbf{m} \in \mathbf{R}_{\mathrm{min}}$ }
	\end{align}
where $ \phi ^ \mathbf{m}:= \prod _{j=1,..., N}\phi _j ^{m_j}$.
Specifically, in \cite{CM20}the following  is proved.
\begin{proposition} \label{lem:generic g}
Let $\displaystyle L=\sup \{  \frac{\| \mathbf{m} \| -1}{2} : \mathbf{m} \in \mathbf{R}_{\mathrm{min}} \} $ and suppose that the operator $H$
satisfies condition \eqref{eq:FGRsimplified}.  Then there exists an open dense
subset $\Omega $ of $\R ^{L}$ s.t.\ if $(g' (0),...., g ^{(L)} (0))\in \Omega $ such that Assumption \ref{ass:FGR} is true   for
\eqref{nls}.
\end{proposition}
  \qed

For $\mathbf{z}=(z_1,\cdots,z_N)\in \C^N$, $\mathbf{m}=(m_1,\cdots,m_N)\in \Z^N$, we define
\begin{align}\label{eq:zkakko}
\mathbf{z}^\mathbf{m}&:=z_1^{(m_1)}\cdots z_N^{(m_N)} \in \C,\text{ where } z^{(m)}:=\begin{cases} z^m & m\geq 0\\ \bar z^{-m} & m<0,\end{cases}\quad \text{ and }\\
\label{eq:defzabs}
|\mathbf{z}|^k&:=(|z_1|^k,\cdots,|z_N|^k)\in \R^N,\ \|\mathbf{z}\|:=\sum |\mathbf{z}|= \sum_{j=1}^N|z_j|\in\R.
\end{align}

We will use the following notation for a ball in a Banach space $B$:
\begin{align}\label{eq:def:ball}
\mathcal{B}_B(u,r):=\{v\in B\ |\ \|v-u\|_B<r\}.
\end{align}

The  Refined Profile  is of the form $\phi(\mathbf{z}) = \mathbf{z} \cdot \boldsymbol{\phi} +o(\|\mathbf{z}\|)$ and is defined by the
following proposition, proved in \cite{CM20}.

\begin{proposition}[Refined Profile]\label{prop:rp}
For any $s \geq 0$, there exist $\delta_s>0$ and $C_s>0$ s.t.\ $\delta_s$ is nonincreasing w.r.t.\ $s\geq 0$ and there exist
\begin{align*}
\{\psi_\mathbf{m}\}_{\mathbf{m}\in \mathbf{NR}_1} &\in C^\infty ( \mathcal{B}_{\R^N}(0,\delta_s^2),(\Sigma^s)^{\sharp \mathbf{NR}_1}), \ \boldsymbol{\varpi}(\cdot) \in C^\infty (\mathcal{B}_{\R^N}(0,\delta_s^2),\R^N) \\
&\text{ and }\mathcal R \in C^\infty(\mathcal{B}_{\C^N}(0,\delta_s),\Sigma^s),
\end{align*}
 s.t.\ $\boldsymbol{\varpi}(0,\cdots,0)=\boldsymbol{\omega}$, $\psi_{\mathbf{m}}(0)=0$ for all $\mathbf{m}\in \mathbf{NR}_1$ and
\begin{align}\label{est:R}
\|\mathcal R(\mathbf{z})\|_{\Sigma^s}\leq C_s \|\mathbf{z} \|^2\sum_{\mathbf{m}\in \mathbf{R}_{\min}} |\mathbf{z}^{\mathbf{m}}|,
\end{align}
where $B_X(a,r):=\{u\in X\ |\ \|u-a\|_X<r\}$, and  if we set
\begin{align}\label{eq:Phianz}
\phi(\mathbf{z}):=\mathbf{z}\cdot\boldsymbol{\phi} + \sum_{\mathbf{m}\in \mathbf{NR}_1}\mathbf{z}^{\mathbf{m}}\psi_{\mathbf{m}}(|\mathbf{z}|^2)\text{ and }z_j(t)=e^{-\im \varpi_j(|\mathbf{z}|^2) t}z_j,
\end{align}
then, setting $\mathbf{z}(t)=(z_1(t),\cdots,z_n(t))$, the function $u(t):=\phi\(\mathbf{z}(t)\)$ satisfies
\begin{align}\label{eq:nlsforce}
\im \partial_t u = H u + g(|u|^2)u -\sum_{\mathbf{m}\in \mathbf{R}_{\mathrm{min}}}\mathbf{z}(t)^{\mathbf{m}} G_{\mathbf{m}} - \mathcal R(\mathbf{z}(t)),
\end{align}
where $\{G_\mathbf{m}\}_{\mathbf{R}_{\min}} \subset\(\Sigma^\infty\)^{\sharp \mathbf{R}_{\min}}$ is given in \eqref{eq:defG}.
Finally, writing $\psi_\mathbf{m}=\psi_\mathbf{m}^{(s)}$, $\boldsymbol{\varpi}=\boldsymbol{\varpi}^{(s)}$ and $\mathcal R=\mathcal R^{(s)}$, for $s_1<s_2$ we have $\psi_\mathbf{m}^{(s_1)}(|\cdot|^2)=\psi_\mathbf{m}^{(s_2)}(|\cdot|^2)$, $\boldsymbol{\varpi}^{(s_1)}(|\cdot|^2)=\boldsymbol{\varpi}^{(s_2)}(|\cdot|^2)$ and $\mathcal R^{(s_1)}=\mathcal R^{(s_2)}$ in $ \mathcal{B}_{\R^N}(0,\delta_{s_2})$.
\end{proposition}
 \qed

We give several formulae related to the refined profile.
Let $X$ be a Banach space and $F\in C^1(\mathcal{B}_{\C^N}(0,\delta),X)$ for some $\delta>0$.
For $\mathbf{z}\in \mathcal{B}_{\C^N}(0,\delta)$ and
 $\mathbf{w}\in \C^N$, we set
 \begin{align*}
 D_{\mathbf{z}}F(\mathbf{z})\mathbf{w}:=\left.\frac{d}{d\epsilon}\right|_{\epsilon=0}F(\mathbf{z}+\epsilon \mathbf{w}).
 \end{align*}
 For $\mathbf{z}(t)$   given by the 2nd equation of \eqref{eq:Phianz},  that is $z_j(t)=e^{-\im \varpi_j(|\mathbf{z}|^2) t}z_j$,      we have
 \begin{align*}
 \im \partial_t \mathbf{z}=\boldsymbol{\varpi}(|\mathbf{z}|^2)\mathbf{z},\ \text{where}\ \boldsymbol{\varpi}(|\mathbf{z}|^2)\mathbf{z}:=(\varpi_1(|\mathbf{z}|^2)z_1,\cdots,\varpi_N(|\mathbf{z}|^2)z_N).
 \end{align*}
 Thus, $\im \partial_t \phi(\mathbf{z}(t))=\im D_{\mathbf{z}}\phi(\mathbf{z}(t))(-\im \boldsymbol{\varpi}(|\mathbf{z}(t)|^2)\mathbf{z}(t))$ and we have the following formula identically satisfied by  $\phi (\mathbf{z})$,
 \begin{align}\label{eq:rfstationary}
\im D_{\mathbf{z}}\phi(\mathbf{z})(-\im \boldsymbol{\varpi}(|\mathbf{z}|^2)\mathbf{z})=H\phi(\mathbf{z})+g(|\phi(\mathbf{z})|^2)\phi(\mathbf{z})-\sum_{\mathbf{R}_{\mathrm{min}}}\mathbf{z}^{\mathbf{m}}G_{\mathbf{m}}-\mathcal{R}(\mathbf{z}).
 \end{align}
 Furthermore, differentiating \eqref{eq:rfstationary} w.r.t.\ $\mathbf{z}$ in  any given  direction $\widetilde{\mathbf{z}}\in \C ^N$, we obtain
 \begin{align}\label{eq:rfderiv}
 H[\mathbf{z}]D_\mathbf{z}\phi(\mathbf{z})\widetilde{\mathbf{z}}=&\im D_{\mathbf{z}}^2\phi (\mathbf{z})   (-\im \boldsymbol{\varpi}(|\mathbf{z}|^2)\mathbf{z},\widetilde{\mathbf{z}})+
 \im D_{\mathbf{z}}\phi(\mathbf{z})\(D_{\mathbf{z}}(-\im \boldsymbol{\varpi}(|\mathbf{z}|^2)\mathbf{z})\widetilde{\mathbf{z}}\)
 \\&\nonumber+\sum_{\mathbf{m}\in \mathbf{R}_{\mathrm{min}}}D_{\mathbf{z}}(\mathbf{z}^{\mathbf{m}})\widetilde{\mathbf{z}}G_{\mathbf{m}}+D_{\mathbf{z}}\mathcal{R}(\mathbf{z})\widetilde{\mathbf{z}},
 \end{align}
 where the operator $H[\mathbf{z}]$  is defined by
 \begin{align}\label{eq:Hz}
 H[\mathbf{z}]f:=Hf+g(|\phi(\mathbf{z})|^2)f+2g'(|\phi(\mathbf{z})|^2) \mathrm{Re}\(\overline{\phi(\mathbf{z})} \ f \)  \phi(\mathbf{z})
 \end{align}
and is selfadjoint for the inner product $\< u, v\>=\Re \int _{\R ^3}u\overline{v}dx$.

The refined profile  $\phi(\mathbf{z})$  contains as a special case the small standing waves bifurcating from the eigenvalues, when they are simple.

\begin{corollary}\label{cor:smallbddst}
Let $s>0$ and $j\in \{1,\cdots,N\}$.
Then,  $\phi\(z(t)\mathbf{e}_j\)$  solves \eqref{nls}  if  $z\in \mathcal B_{\C}(0,\delta_s)$ and  $z(t)=e^{-\im \varpi_j(|z\mathbf{e}_j|^2)t}z$.

\end{corollary}

\begin{proof}
Since $(z\mathbf{e}_j)^{\mathbf{m}}=0$ for $\mathbf{m}\in \mathbf{R}_{\mathrm{min}}$, we see that from \eqref{est:R} and \eqref{eq:nlsforce} the remainder terms $\sum_{\mathbf{m}\in \mathbf{R}_{\mathrm{min}}}\mathbf{z}(t)^{\mathbf{m}} G_{\mathbf{m}} + \mathcal R(\mathbf{z}(t))$ are $0$ in \eqref{eq:nlsforce}.
Therefore, we have the conclusion.
\end{proof}

\begin{remark}
	If the eigenvalues of $H$ are not simple the above   does not hold anymore in general. See Gustafson-Phan \cite{GP11SIMA}.
\end{remark}

We call solitons, or standing waves, the functions
\begin{align}\label{eq:defbddst}
\phi_j(z):=\phi(z \mathbf{e}_j)  \text{ for $z\in \mathcal{B}_{\C}  (0, \delta _s)$.}
\end{align}

%
%
%

The main result, which have first proved in \cite{CM15APDE}   is the following.
\begin{theorem}\label{thm:main}
Under the Assumptions \ref{ass:nonres}, \ref{ass:linearInd} and \ref{ass:FGR}, there exist $\delta_0>0$ and $C>0$ s.t.\ for all $u_0\in H^1$ with $\epsilon_0:=\|u_0\|_{H^1}< \delta_0$, there exists $j\in \{1,\cdots,N\}$, $z\in C^1(\R,\C)$, $\eta_+\in H^1$ and $\rho_+\geq 0$ s.t.
\begin{align}
\lim_{t\to \infty} \|u(t)- \phi_j(z(t)) - e^{\im t \Delta }\eta_+ \|_{H^1}=0, \label{eq:limit_1}
\end{align}
with  $C^{-1}\epsilon_0^2\leq \rho_+^2 + \|\eta_+\|_{H^1}^2 \leq  C \epsilon_0^2$ and
\begin{align}&
\lim_{t\to +\infty}|z (t)|=\rho_+  \ .  \label{eq:limit_2}
\end{align}
\end{theorem}

When written in  the Modulation parameters, the NLS appears like a complicated system where some discrete modes are coupled to radiation. The discrete modes tend to produce complicated patterns, similar to the ones of a linear system with eigenvalues. However, asymptotically in time the nonlinear interaction is responsible of spilling of energy into radiation which disperses at space infinity and to the selection of a unique nonlinear standing wave.
 Theorem \ref{thm:main}  is the same of the main theorem in \cite{CM20}  and is very similar to the main theorem in \cite{CM15APDE}.
 The proofs here and in \cite{CM20} are much simpler than in  \cite{CM15APDE} or in earlier papers containing early partial results, like
  \cite{SW04RMP,TY02ATMP}. In \cite{CM15APDE}, in order to detect the nonlinear redistribution of the energy it was necessary to make full use of the hamiltonian structure of our NLS, by first introducing    Darboux coordinates and by then considering a normal forms argument. The discovery of the notion of Refined Profile made in  \cite{Maeda17SIMA} and its further development  in  \cite{CM20} allows to forgo the normal forms argument because an almost optimal system of coordinates is provided automatically by the Refined Profile. In \cite{CM20} we  introduced   Darboux coordinates in a way much simpler than in \cite{CM15APDE}. Undoubtedly, Darboux coordinates are  quite natural for a Hamiltonian system and in \cite{CM20} they contribute to simplify the system.  In the present note however, we provide a different proof which, except for the information that mass and energy are constant, thus guaranteeing the global existence of our small $H^1$ solutions, does not make explicit use of the hamiltonian structure of the equations.

\section{The proof}\label{sec:prmain}

We start from constructing the modulation coordinate.
First, we have the following.
\begin{lemma}\label{lem:lincor}
There exist  $\delta>0$ and $\mathbf z\in C^\infty (   \mathcal{B}_{  \Sigma^{-1}}  (0, \delta)  , \C^N)$ s.t.
\begin{align*}
\forall \widetilde{\mathbf{z}}\in \C^N,\ \<\im\(u-\phi(\mathbf{z}(u))\), D_{\mathbf{z}}\phi(\mathbf{z}(u))\widetilde{\mathbf{z}}\>=0.
\end{align*}
\end{lemma}

\begin{proof}
Standard.
\end{proof}

We set
\begin{align}
\eta(u):=u-\phi(\mathbf{z}(u)).
\end{align}
In the following we write $\eta=\eta(u)$ and $\mathbf{z}=\mathbf{z}(u)$.
Substituting $u=\phi(\mathbf{z})+\eta$ to \eqref{nls} and using \eqref{eq:rfstationary}, we have
\begin{align}\label{eq:modnls}
\im \partial_t \eta + \im D_{\mathbf{z}}\phi(\mathbf{z})\(\partial_t \mathbf{z} + \im \boldsymbol{\varpi}(|\mathbf{z}|^2)\mathbf{z}\)=H[\mathbf{z}]\eta + \sum_{\mathbf{R}_{\mathrm{min}}}\mathbf{z}^{\mathbf{m}}G_{\mathbf{m}} +\mathcal{R}(\mathbf{z})+F(\mathbf{z},\eta),
\end{align}
where
\begin{align*}
F(\mathbf{z},\eta)=g(|\phi(\mathbf{z})+\eta|^2)(\phi(\mathbf{z})+\eta)-g(|\phi(\mathbf{z})|^2)\phi(\mathbf{z})-g(|\phi(\mathbf{z})|^2)\eta -2g'(|\phi(\mathbf{z})|^2)\mathrm{Re}\(\overline{\phi(\mathbf{z})} \eta\) \phi(\mathbf{z}).
\end{align*}
Given an interval $I\subseteq \R$ we set
\begin{align}
\mathrm{Stz}^j(I):=L^\infty_t  H^j (I)   \cap L^2_t  W^{j,6}(I),\quad \mathrm{Stz}^{*j}(I):=L^1_t  H^j(I )  +  L^2_t W^{j,6/5}(I ),\ j=0,1,\label{strich_sp}
\end{align}
where $H^0=L^2$ and $W^{0,p}=L^p$ and use  Yajima's \cite{Y1} Strichartz inequalities, for $t_0\in \overline{I}$,
\begin{align}
\|e^{-\im t H}P_c v\|_{\mathrm{Stz}^j(\R)}\lesssim \|v\|_{H^j},\ \|\int _{t_0}^t e^{-\im(t-s)H}P_cf(s)\,ds\|_{\mathrm{Stz}^j(I)}\lesssim \|f\|_{\mathrm{Stz}^{*j}(I)}, \ j=0,1 .\label{strich_est}
\end{align}
Under the assumptions of Theorem \ref{thm:main}   we have $\|u\|_{L^\infty H^1(\R)}\lesssim \epsilon_0$ from energy and mass conservation.
Since  $\|u\|_{H^1}\sim \|\mathbf{z}\|+\|\eta\|_{H^1}$, we conclude
\begin{align*}
\|\mathbf{z}\|_{L^\infty_t(\R )} + \|\eta\|_{L^\infty_t H^1(\R )}\lesssim \epsilon_0.
\end{align*}


 \begin{theorem}[Main Estimates]\label{thm:mainbounds}
There exist $\delta _0>0$ and $C_0>0$ s.t.\ if $\epsilon_0=\|u_0\|_{H^1}< \delta_0$, we have
\begin{align}
   \|  \eta \| _{\mathrm{Stz}^1(I)} +\sum_{\mathbf{m}\in \mathbf{R}_{\mathrm{min}} }\|  \mathbf{z}^{\mathbf{m}} \| _{L^2_t(I)} + \|\partial_t \mathbf{z}+\im \boldsymbol{\varpi}(|\mathbf{z}|^2)\mathbf{z}\|_{L^2_t(I)}&\le
  C   \epsilon_0,
  \label{Strichartzradiation}
\end{align}
      for $I= [0,\infty )$ and $C=C_0$.
\end{theorem}
Notice that \eqref{Strichartzradiation}, the  equation \eqref{eq:modnls}  satisfied by $\eta$,  estimate  \eqref{est:R} for $\mathcal{R}(\mathbf{z})$    and   Lemma \ref{lem:estF} below  for $F(\mathbf{z},\eta )$, allow to prove in a standard and elementary   fashion  that
$\eta (t)$    scatters as $t\to +\infty$, i.e. there exists $\eta _+\in H^1$ such that $ \| \eta (t) - e^{\im t\Delta }\eta _+\| _{H^1}\xrightarrow{t\to +\infty }0$. From
\eqref{Strichartzradiation} we have $\| \eta _+ \| _{H^1}\le  C   \epsilon_0$.

 \noindent Using mass conservation we have
\begin{align*}&
     \|   \phi (\mathbf{z} (t)) \|_{L^2}  ^{ 2}  =   \|   u_0 \|_{L^2}  ^{ 2}  - 2\< \phi (\mathbf{z}(t)), e^{\im t\Delta }\eta _+ \> -   2\< \phi (\mathbf{z}(t)), \eta (t)  - e^{\im t\Delta }\eta _+ \> - \|  \eta (t)  \|_{L^2}  ^{ 2}\\& \xrightarrow{t\to +\infty }\|   u_0 \|_{L^2}  ^{ 2}-  \|  \eta _+  \|_{L^2}  ^{ 2}.
\end{align*}
So, by $ \|   \phi (\mathbf{z} (t)) \|_{L^2}  ^{ 2}  = \| \mathbf{z} (t) \| ^2 +o(\| \mathbf{z} (t) \| ^2)$,  we get $\displaystyle \lim _{t\to +\infty}\| \mathbf{z} (t) \| ^2 =\rho _+ ^2$
for some $0\le \rho _+\le 2  C   \epsilon_0$.

The fact that $\mathbf{z} ^{\mathbf{m}} \in L^2(\R _+ )$ and, as it is easy to see, $\partial _t (\mathbf{z} ^{\mathbf{m}}) \in L^\infty(\R _+ )\cap C^0([0,\infty )$, imply
 $\mathbf{z} ^{\mathbf{m}} \xrightarrow{t\to +\infty }0 $  for any $  \mathbf{m}\in \mathbf{R}_{\mathrm{min}}$. This implies $z_k \xrightarrow{t\to +\infty }0 $ for all $k$ except at most for one,
  yielding  the selection of one coordinate $j$ in the statement of
Theorem \ref{thm:main}.
The proof that Theorem \ref{thm:mainbounds}  implies Theorem \ref{thm:main} is like in \cite{CM15APDE}.

By completely routine arguments  discussed in \cite{CM15APDE},
\eqref{Strichartzradiation} for  $I= [0,\infty )$ is a consequence of the following Proposition.

\begin{proposition}\label{prop:mainbounds} There exists  a  constant $c_0>0$ s.t.\
for any  $C_0>c_0$ there is a value    $\delta _0= \delta _0(C_0)   $ s.t.\  if    \eqref{Strichartzradiation}
holds  for $I=[0,T]$ for some $T>0$, for $C=C_0$  and for $u_0\in B_{H^1}(0,\delta_0)$,
then in fact for $I=[0,T]$  the inequalities  \eqref{Strichartzradiation} holds  for   $C=C_0/2$.
\end{proposition}
In the remainder of the paper   we prove {Proposition} \ref{prop:mainbounds}.

\subsection{Estimate of the continuous variable $\eta$}

In the following, we set $\epsilon_0=\|u_0\|_{H^1}$.
Further, when we use $\lesssim$, the implicit constant will not depend on $C_0$.
We start from the estimate of the remainder term $F$.
\begin{lemma}\label{lem:estF}
Under the assumption of Proposition \ref{prop:mainbounds}, we have
\begin{align}
\|F(\mathbf{z},\eta)\|_{\mathrm{Stz}^{*1}(I)}\lesssim C_0\epsilon_0^3.
\end{align}
\end{lemma}

\begin{proof}
By \eqref{eq:ggrowth}, we have the pointwise bound
\begin{align}
|F(\mathbf{z},\eta)|+|\nabla_x F(\mathbf{z},\eta)|\lesssim \(1+|\eta|^2\)\(|\phi(\mathbf{z})|+|\nabla_x \phi(\mathbf{z})| +|\eta|\)|\eta|\(|\eta|+|\nabla_x \eta|\).
\end{align}
Using this, we obtain  the conclusion by H\"older and Sobolev estimates.
\end{proof}

We set
\begin{align}
\mathcal{H}_c[\mathbf{z}]:=\{v\in L^2\ |\ \forall \widetilde{\mathbf{z}}\in \C^N,\ \<\im v, D_{\mathbf{z}}\phi(\mathbf{z})\widetilde{\mathbf{z}}\>=0\}.
\end{align}
Notice that for $u\in H^1$, $\eta(u)\in \mathcal{H}_c[\mathbf{z}(u)]\cap H^1$.
Following Gustafson, Nakanishi and Tsai \cite{GNT}, we can construct an inverse of $P_c$ on $\mathcal{H}_c[\mathbf{z}]$.

\begin{lemma}\label{lem:GNTR}
There exists $\delta>0$ s.t.\ there exists $\{a_{jA}\}_{j=1,\cdots,N,A=\mathrm{R},\mathrm{I}}\in C^\infty(\mathcal{B}_{\C^N(0,\delta)}, \Sigma^1)$ s.t.
\begin{align}\label{eq:estRa}
\|a_{jA}(\mathbf{z})\|_{\Sigma^1}\lesssim \|\mathbf{z}\|^2, \ j=1,\cdots,N,\ A=\mathrm{R},\mathrm{I}
\end{align}
and
\begin{align}
R[\mathbf{z}]:=\mathrm{Id}-\sum_{j=1}^N\(\<\cdot,a_{j\mathrm{R}}(\mathbf{z})\>\phi_j+\<\cdot,a_{j\mathrm{R}}(\mathbf{z})\>\im \phi_j\),
\end{align}
satisfies $\left.R[\mathbf{z}]P_c\right|_{\mathcal{H}_c[\mathbf{z}]}=\left.\mathrm{Id}\right|_{\mathcal{H}_c[\mathbf{z}]}$, $\left.P_cR[\mathbf{z}] \right|_{P_c L^2}=\left.\mathrm{Id} \right|_{P_c L^2}$.
\end{lemma}

\begin{proof}
A proof is in \cite{CM15APDE}.
\end{proof}

We set $\widetilde{\eta}=P_c\eta$.
By Lemma \ref{lem:GNTR}, we have $\eta=R[z]\widetilde{\eta}$   and  $\|\eta\|_{\mathrm{Stz}^1}\sim \|\widetilde{\eta}\|_{\mathrm{Stz}^1}$.
Applying $P_c$ to \eqref{eq:modnls}, we have
\begin{align}\label{eq:tildeeta}
\im \partial_t \widetilde{\eta}=&H\widetilde{\eta}-\im P_c D_{\mathbf{z}}\phi(\mathbf{z})\(\partial_t \mathbf{z}+\im \boldsymbol{\varpi}(|\mathbf{z}|^2)\mathbf{z}\)+\sum_{\mathbf{m}\in \mathbf{R}_{\mathrm{min}}}\mathbf{z}^{\mathbf{m}}P_c G_{\mathbf{m}}\\& +P_c\mathcal{R}(\mathbf{z})+P_cF(\mathbf{z},\eta)+P_c\(H[\mathbf{z}]-H\)\eta.\nonumber
\end{align}

\begin{lemma}\label{est_eta}
Under the assumption of Proposition \ref{prop:mainbounds}, we have
\begin{align}\label{est_eta1}
\|\eta\|_{\mathrm{Stz}^1(I)}\lesssim
\epsilon_0+C(C_0)\epsilon_0^3  + \sum_{\mathbf{m}\in \mathbf{R}_{\mathrm{min}}}\|\mathbf{z}^{\mathbf{m}}\|_{L^2(I)}.
\end{align}
\end{lemma}
\proof Obviously, from $\|\eta\|_{\mathrm{Stz}^1}\sim \|\widetilde{\eta}\|_{\mathrm{Stz}^1}$ it is enough  to bound the latter. By
 Strichartz estimates \eqref{strich_est} and Lemma \ref{lem:estF} we easily obtain
\begin{align*}
\|\widetilde{\eta}\|_{\mathrm{Stz}^1(I)}\lesssim
\epsilon_0+C(C_0)\epsilon_0^3 + \|   P_c D_{\mathbf{z}}\phi(\mathbf{z})\(   \partial_t \mathbf{z}+\im \boldsymbol{\varpi}(|\mathbf{z}|^2)\mathbf{z}\) \|_{L^2(I)}+ \sum_{\mathbf{m}\in \mathbf{R}_{\mathrm{min}}}\|\mathbf{z}^{\mathbf{m}}\|_{L^2(I)}.
\end{align*}
Using the fact that   $\|P_c D_{\mathbf{z}}\phi(\mathbf{z})\|_{\Sigma^1}=O\(\|\mathbf{z}\| ^2\) $, we obtain \eqref{est_eta1}.

\qed

We set $Z(\mathbf{z}):=- \sum_{\mathbf{m}\in \mathbf{R}_{\min}} \mathbf{z}^{\mathbf{m}}R_+(\mathbf{m}\cdot \boldsymbol{\omega})P_c G_{\mathbf{m}}$ and $\xi :=\widetilde{\eta}+Z$, where $R_+(\lambda):=(H-\lambda-\im 0)^{-1}$.
Using the identity
\begin{align}\label{eq:difzm}
\( D_{\mathbf{z}}\mathbf{z}^{\mathbf{m}}\)(\im \boldsymbol{\omega}\mathbf{z})=\im \mathbf{m}\cdot \boldsymbol{\omega}\mathbf{z}
\end{align}
with  $\boldsymbol{\omega}\mathbf{z}:=(\omega_1z_1,\cdots,\omega_Nz_N)$, we see that  $Z$ satisfies
\begin{align}\label{eq:Z}
&-\im \partial_t Z(\mathbf{z}) + HZ(\mathbf{z}) = \sum_{\mathbf{m}\in \mathbf{R}_{\mathrm{min}}}\mathbf{z}^{\mathbf{m}}P_cG_{\mathbf{m}}+\mathcal{R}_Z(\mathbf{z}),
\end{align}
where
\begin{align}
 \mathcal{R}_Z(\mathbf{z})=\im \sum_{\mathbf{m}\in \mathbf{R}_{\mathrm{min}}}  D_{\mathbf{z}}  \(\mathbf{z}^{\mathbf{m}}\) \left [  \(\partial_t \mathbf{z}+\im \boldsymbol{\varpi}(|z|^2) \mathbf{z}\) +         \(  \im \boldsymbol{\omega}\mathbf{z}-\im \boldsymbol{\varpi}(|z|^2) \mathbf{z} \) \right ]R_+(\mathbf{m}\cdot \boldsymbol{\omega})P_c G_{\mathbf{m}}.\nonumber
\end{align}
Substituting $\widetilde{\eta}=\xi-Z(\mathbf{z})$ into \eqref{eq:tildeeta}, we obtain
\begin{align}\label{eq:xi}
\im \partial_t \xi=&H\xi-\im P_c D_{\mathbf{z}}\phi(\mathbf{z})   \(\partial_t \mathbf{z}+\im \boldsymbol{\varpi}(|\mathbf{z}|^2)\mathbf{z}\) +P_c\mathcal{R}(\mathbf{z})+P_cF(\mathbf{z},\eta)+P_c\(H[\mathbf{z}]-H\)\eta+\mathcal{R}_{Z}(\mathbf{z}).
\end{align}

\begin{lemma}\label{lem:estxi}
Under the assumption of Proposition \ref{prop:mainbounds}, we have
\begin{align*}
\|\xi\|_{L^2\Sigma^{0-}(I)}\lesssim \epsilon_0+C_0\epsilon_0^3.
\end{align*}
\end{lemma}

\begin{proof}
By $\|\cdot\|_{L^2\Sigma^{0-}}\lesssim \|\cdot\|_{\mathrm{Stz}^0}$ and Strichartz estimates \eqref{strich_est}, we have
\begin{align}\label{eq:estxi1}
\|\xi\|_{L^2\Sigma^{0-}}&\leq  \|\widetilde{\eta}(0)\|_{L^2}+\|e^{-\im t H}Z(\mathbf{z}(0))\|_{L^2\Sigma^{0-}}+\|\int_0^t e^{-\im (t-s)H}\mathcal{R}_Z(\mathbf{z}(u(s)))\,ds\|_{L^2\Sigma^{0-}} \\& +
\|\im P_c D_{\mathbf{z}}\phi(\mathbf{z})\(    \dot {\mathbf{z}}+\im \boldsymbol{\varpi}(|\mathbf{z}|^2)\mathbf{z}\)  -\mathcal{R}(\mathbf{z})-F(\mathbf{z},\eta)-\(H[\mathbf{z}]-H\)\eta\|_{\mathrm{Stz}^{*0}}\nonumber
\end{align}
where $\mathbf{z}(t)=\mathbf{z}(u(t))$.
One can bound the contribution of the 2nd line of \eqref{eq:estxi1} by $\lesssim C(C_0)\epsilon_0^3 $ using, as in Lemma \ref{est_eta}, $\|P_c D_{\mathbf{z}}\phi(\mathbf{z})\|_{\Sigma^1}=O\(\|\mathbf{z}\| ^2\) $ and $ |  D_{\mathbf{z}}  \(\mathbf{z}^{\mathbf{m}}\)  \(  \im \boldsymbol{\omega}\mathbf{z}-\im \boldsymbol{\varpi}(|z|^2) \mathbf{z} \)      |  = O\(\|\mathbf{z}\| ^2\) |\mathbf{z}^{\mathbf{m}}|      $  by \eqref{eq:difzm}.
Similarly, the first term in the r.h.s.\ of \eqref{eq:estxi1} can be bounded by $\lesssim
\epsilon_0$.
For the 2nd and 3rd term in the r.h.s.\ of \eqref{eq:estxi1}, we  use the estimate
\begin{align}\label{eq:outgoing}
\|e^{-\im t H}R_+(\mathbf{m}\cdot \boldsymbol{\omega})P_cf\|_{\Sigma^{0-}}\lesssim\<t\>^{-3/2}\|f\|_{\Sigma^0}.
\end{align}
Using \eqref{eq:outgoing}, we have
\begin{align*}
\|e^{-\im t H}Z(\mathbf{z}(0))\|_{L^2\Sigma^{0-}(I)}\lesssim \sum_{\mathbf{m}\in \mathbf{R}_{\mathrm{min}}}|\mathbf{z}(0)^{\mathbf{m}}|\|\<t\>^{-3/2}\|_{L^2}\|G_{\mathbf{m}}\|_{\Sigma^0}\lesssim \epsilon_0,
\end{align*}
and
\begin{align*}
&\|\int_0^t e^{-\im (t-s)H}\mathcal{R}_Z(\mathbf{z}(u(s)))\,ds\|_{L^2\Sigma^{0-}(I)}\\&\leq \sum_{\mathbf{m}\in \mathbf{R}_{\mathrm{min}}}\| \int_0^t \left|D_{\mathbf{z}}\(\mathbf{z}^{\mathbf{m}}(s)\)\(\partial_t \mathbf{z}(s)+\im \boldsymbol{\omega}\mathbf{z}(s)\)\right| \|e^{-\im (t-s)H}R_+(\mathbf{m}\cdot \boldsymbol{\omega})P_cG_{\mathbf{m}}\|_{\Sigma^{0-}}\,ds \|_{L^2(I)}\\&
\lesssim \sum_{\mathbf{m}\in \mathbf{R}_{\mathrm{min}}} \| \epsilon_0^2 \int_0^t\( |\partial_t \mathbf{z}(s)+\im \boldsymbol{\varpi}(|\mathbf{z}|^2(s))\mathbf{z}(s)| + |\mathbf{z}(s)^{\mathbf{m}}|\) \<t-s\>^{-3/2}\|_{L^2(I)}\lesssim C(C_0)\epsilon_0^3,
\end{align*}
where we have used $\partial_t \mathbf{z}+\im \boldsymbol{\omega}\mathbf{z}=\partial_t \mathbf{z}+\im \boldsymbol{\omega}(|\mathbf{z}|^2)\mathbf{z} -\im\( \boldsymbol{\omega}(|\mathbf{z}|^2)\mathbf{z}- \boldsymbol{\omega}\mathbf{z}\)$ and \eqref{eq:difzm} in the 2nd inequality and Young's convolution inequality in the 3rd inequality.
Therefore, we have the conclusion.
\end{proof}

\subsection{Estimate of discrete variables}

We next estimate the quantities $\|\partial_t \mathbf{z}+\im \boldsymbol{\varpi}(|\mathbf{z}|^2)\mathbf{z}\|_{L^2}$ and $\sum_{\mathbf{m}\in \mathbf{R}_{\mathrm{min}}}\|\mathbf{z}^{\mathbf{m}}\|_{L^2}$.
To do so, we first compute the inner product  $\<\eqref{eq:modnls},D_{\mathbf{z}}\phi(\mathbf{z})\widetilde{\mathbf{z}}\>$ for any given $\widetilde{\mathbf{z}}\in \C^N$.
First, notice that by $\eta\in \mathcal{H}_c[\mathbf{z}]$ we obtain the orthogonality relation
\begin{align*}
\<\im \partial_t \eta, D_{\mathbf{z}}\phi(\mathbf{z})\widetilde{\mathbf{z}}\>=-\<\im \eta, D_{\mathbf{z}}^2\phi(\mathbf{z})(\partial_t \mathbf{z},\widetilde{\mathbf{z}})\>.
\end{align*}
 Second, applying the inner product  $\< \eta , \cdot \>$  to equation \eqref{eq:rfderiv},   we have
\begin{align*}
\<H[\mathbf{z}]\eta,D_{\mathbf{z}}\phi(\mathbf{z})\widetilde{\mathbf{z}}\>=\<\im \eta,D_{\mathbf{z}}^2\phi(\mathbf{z})(\boldsymbol{\varpi}(|\mathbf{z}|^2)\mathbf{z},\widetilde{\mathbf{z}})\>+\sum_{\mathbf{m}\in \mathbf{R}_{\mathrm{min}}}\<\eta,\(D_{\mathbf{z}}\(\mathbf{z}^{\mathbf{m}}\)\widetilde{\mathbf{z}}\)G_{\mathbf{m}}\>+\<\eta,D_{\mathbf{z}}\mathcal{R}(\mathbf{z})\widetilde{\mathbf{z}}\>,
\end{align*}
where we exploited the selfadjointness of $H[\mathbf{z}]$  and the orthogonality in Lemma \ref{lem:lincor}.
Thus,  applying $\< \cdot  , D_{\mathbf{z}}\phi(\mathbf{z})\widetilde{\mathbf{z}}\>$ to  equation \eqref{eq:modnls} for $\eta$ and  using the last two equalities, we obtain
\begin{align}\label{eq:discfund}
  \<\im D_{\mathbf{z}}\phi(\mathbf{z})(\partial_t \mathbf{z}+\im \boldsymbol{\varpi}(|\mathbf{z}|^2)\mathbf{z}),D_{\mathbf{z}}\phi(\mathbf{z})\widetilde{\mathbf{z}}\> =\<\im \eta, D_{\mathbf{z}}^2\phi(\mathbf{z})\(\partial_t \mathbf{z}+\im \boldsymbol{\varpi}(|\mathbf{z}|^2)\mathbf{z},\widetilde{\mathbf{z}}\)\>
+\<\eta,D_{\mathbf{z}}\mathcal{R}(\mathbf{z})\widetilde{\mathbf{z}}\>\\ +\sum_{\mathbf{m}\in \mathbf{R}_{\mathrm{min}}}\<\eta,\(D_{\mathbf{z}}\(\mathbf{z}^{\mathbf{m}}\)\widetilde{\mathbf{z}}\)G_{\mathbf{m}}\>+\sum_{\mathbf{m}\in \mathbf{R}_{\mathrm{min}}}\<\mathbf{z}^{\mathbf{m}}G_{\mathbf{m}},D_{\mathbf{z}}\phi(\mathbf{z})\widetilde{\mathbf{z}}\>\nonumber \\ +\<\mathcal{R}(\mathbf{z}),D_{\mathbf{z}}\phi(\mathbf{z})\widetilde{\mathbf{z}}\>+\<F(\mathbf{z},\eta),D_{\mathbf{z}}\phi(\mathbf{z})\widetilde{\mathbf{z}}\>.\nonumber
\end{align}

Using $\widetilde{\mathbf{z}}=\mathbf{e}_j, \im \mathbf{e}_j$ we have the following.
\begin{lemma}\label{lem:zpres}
Under the assumption of Proposition \ref{prop:mainbounds}, we have
\begin{align}
\partial_t z_j+\im \varpi_j(|\mathbf{z}|^2)z_j=-\im \sum_{\mathbf{m}\in \mathbf{R}_{\mathrm{min}}}\mathbf{z}^{\mathbf{m}}\<G_{\mathbf{m}},\phi_j\>+r_j(\mathbf{z},\eta),\label{equat_z}
\end{align}
where $r_j(\mathbf{z},\eta)$ satisfies
\begin{align*}
\|r_j(\mathbf{z},\eta)\|_{L^2(I)}\lesssim C_0\epsilon_0^3.
\end{align*}
In particular, we have
\begin{align}
\|\partial_t \mathbf{z}+\im \boldsymbol{\varpi}(|\mathbf{z}|^2)\mathbf{z}\|_{L^2(I)}\lesssim \|\mathbf{z}^{\mathbf{m}}\|_{L^2(I)}+C_0\epsilon_0^3.\label{estimate_z}
\end{align}
\end{lemma}

\begin{proof}
First
since $D_{\mathbf{z}}\phi(0)\widetilde{\mathbf{z}}=\widetilde{\mathbf{z}}\cdot \boldsymbol{\phi}$, we have
\begin{align}\label{eq:zj1}
\<\im D_{\mathbf{z}}\phi(\mathbf{z})(\partial_t \mathbf{z}+\im \boldsymbol{\varpi}(|\mathbf{z}|^2)\mathbf{z}),D_{\mathbf{z}}\phi(\mathbf{z})\widetilde{\mathbf{z}}\>=\sum_{j=1}^N\Re(\im(\partial_t z_j+\im \varpi_j(|\mathbf{z}|^2)z_j)\overline{\widetilde{z}_j})+r(\mathbf{z},\widetilde{\mathbf{z}}),
\end{align}
where
\begin{align}\label{eq:zj2}
r(\mathbf{z},\widetilde{\mathbf{z}})=&\<\im \(D_{\mathbf{z}}\phi(\mathbf{z})-D_{\mathbf{z}}\phi(0)\)(\partial_t \mathbf{z}+\im \boldsymbol{\varpi}(|\mathbf{z}|^2)\mathbf{z}),D_{\mathbf{z}}\phi(\mathbf{z})\widetilde{\mathbf{z}}\>\\&+\<\im D_{\mathbf{z}}\phi(0)(\partial_t \mathbf{z}+\im \boldsymbol{\varpi}(|\mathbf{z}|^2)\mathbf{z}),\(D_{\mathbf{z}}\phi(\mathbf{z})-D_{\mathbf{z}}\phi(0)\)\widetilde{\mathbf{z}}\>.\nonumber
\end{align}
Since $\|D_{\mathbf{z}}\phi(\mathbf{z})-D_{\mathbf{z}}\phi(0)\|_{L^2}\lesssim |\mathbf{z}|^2\lesssim \epsilon_0^2$,   by the assumptions of Proposition \ref{prop:mainbounds}  we have
\begin{align}\label{eq:zj4}
\|r(\mathbf{z},\widetilde{\mathbf{z}})\|_{L^2(I)}\lesssim C_0\epsilon_0^3 \text{  for all }  \widetilde{\mathbf{z}}=\mathbf{e}_1,\im \mathbf{e}_1,\cdots,\mathbf{e}_N,\im \mathbf{e}_N.
\end{align}
Setting
\begin{align}\label{eq:zj3}
\widetilde{r}(\mathbf{z},\widetilde{\mathbf{z}},\eta):=&\<\im \eta, D_{\mathbf{z}}^2\phi(\mathbf{z})\(\partial_t \mathbf{z}+\im \boldsymbol{\varpi}(|\mathbf{z}|^2)\mathbf{z},\widetilde{\mathbf{z}}\)\>
+\<\eta,D_{\mathbf{z}}\mathcal{R}(\mathbf{z})\widetilde{\mathbf{z}}\>+\sum_{\mathbf{m}\in \mathbf{R}_{\mathrm{min}}}\<\eta,\(D_{\mathbf{z}}\(\mathbf{z}^{\mathbf{m}}\)\widetilde{\mathbf{z}}\)G_{\mathbf{m}}\>\\&+\sum_{\mathbf{m}\in \mathbf{R}_{\mathrm{min}}}\<\mathbf{z}^{\mathbf{m}}G_{\mathbf{m}},\(D_{\mathbf{z}}\phi(\mathbf{z})-D_{\mathbf{z}}\phi(0)\)\widetilde{\mathbf{z}}\>+\<\mathcal{R}(\mathbf{z}),D_{\mathbf{z}}\phi(\mathbf{z})\widetilde{\mathbf{z}}\>+\<F(\mathbf{z},\eta),D_{\mathbf{z}}\phi(\mathbf{z})\widetilde{\mathbf{z}}\>,\nonumber
\end{align}
  by the assumptions of Proposition \ref{prop:mainbounds}  we have
\begin{align}\label{eq:zj5}
\|\widetilde{r}(\mathbf{z},\widetilde{\mathbf{z}},\eta)\|_{L^2(I)}\lesssim C_0 \epsilon_0^3    \text{  for all }  \widetilde{\mathbf{z}}=\mathbf{e}_1,\im \mathbf{e}_1,\cdots,\mathbf{e}_N,\im \mathbf{e}_N.
\end{align}
Therefore, since $D\phi(0)\im^k\mathbf{e}_j=\im^k \phi_j$ ($k=0,1$), we have
\begin{align*}
-\mathrm{Im}\(\partial_t z_j +\im \varpi_j(|\mathbf{z}|^2)z_j\)&=\sum_{\mathbf{m}\in \mathbf{R}_{\mathrm{min}}}\<\mathbf{z}^{\mathbf{m}}G_{\mathbf{m}},\phi_j\>-r(\mathbf{z},\mathbf{e}_j)+\widetilde{r}(\mathbf{z},\mathbf{e}_j,\eta),\\
\mathrm{Re}\(\partial_t z_j + \im  \varpi_j(|\mathbf{z}|^2)z_j\)&=\sum_{\mathbf{m}\in \mathbf{R}_{\mathrm{min}}}\<\mathbf{z}^{\mathbf{m}}G_{\mathbf{m}},\im\phi_j\>-r(\mathbf{z},\im\mathbf{e}_j)+\widetilde{r}(\mathbf{z},\im\mathbf{e}_j,\eta).
\end{align*}
Since $G_{\mathbf{m}}$ (as can be seen from the proof in Sect. \cite{CM20}) and $\phi_j$ are $\R$-valued, we have
\begin{align*}
\partial_t z_j +  \im \varpi_j(|\mathbf{z}|^2)z_j=-\im \sum_{\mathbf{m}}\<G_{\mathbf{m}},\phi_j\> \mathbf{z}^{\mathbf{m}} -r(\mathbf{z},\im \mathbf{e}_j)+\im r(\mathbf{z},\mathbf{e}_j)+\widetilde{r}(\mathbf{z},\im \mathbf{e}_j,\eta)-\im \widetilde{r}(\mathbf{z},\mathbf{e}_j,\eta).
\end{align*}
Therefore, from
 \eqref{eq:zj4} and
  \eqref{eq:zj5}, we have the conclusion with
  $r_j(\mathbf{z},\eta)=-r(\mathbf{z},\im \mathbf{e}_j)+\im r(\mathbf{z},\mathbf{e}_j)+\widetilde{r}(\mathbf{z},\im \mathbf{e}_j,\eta)-\im \widetilde{r}(\mathbf{z},\mathbf{e}_j,\eta)$.
\end{proof}

Having estimated $\eta$ and $\partial_t \mathbf{z}+\im \boldsymbol{\varpi}(|\mathbf{z}|^2)\mathbf{z}$ in terms of
   $\sum_{\mathbf{m}\in \mathbf{R}_{\mathrm{min}}}\|\mathbf{z}^{\mathbf{m}}\|_{L^2(I)}$,  we need to estimate the latter quantity. Here we use the Fermi Golden Rule.

\begin{lemma}
Under the assumption of Proposition \ref{prop:mainbounds}, we have
\begin{align}
\sum_{\mathbf{m}\in \mathbf{R}_{\mathrm{min}} }\|\mathbf{z}^{\mathbf{m}}\|_{L^2}\lesssim \epsilon_0+(C_0\epsilon_0)\epsilon_0.\label{FGReq}
\end{align}
\end{lemma}

\begin{proof}
We substitute $\widetilde{\mathbf{z}}=\im \boldsymbol{\varpi}(|\mathbf{z}|^2)\mathbf{z}$ into \eqref{eq:discfund} and we make various simplifications.
The first, by  $\<f,\im f\>=0$, is
\begin{align}\label{eq:lFGR2}
\<\im D_{\mathbf{z}}\phi(\mathbf{z})(\partial_t \mathbf{z}+\im \boldsymbol{\varpi}(|\mathbf{z}|^2)\mathbf{z}),D_{\mathbf{z}}\phi(\mathbf{z})\im \boldsymbol{\varpi}(|\mathbf{z}|^2)\mathbf{z}\>=\<\im D_{\mathbf{z}}\phi(\mathbf{z})(\partial_t \mathbf{z}),D_{\mathbf{z}}\phi(\mathbf{z})\im \boldsymbol{\varpi}(|\mathbf{z}|^2)\mathbf{z}\>.
\end{align}
Next,
\begin{align}\nonumber
\<\sum_{\mathbf{m}\in \mathbf{R}_{\mathrm{min}}}\mathbf{z}^{\mathbf{m}}G_{\mathbf{m}}+\mathcal{R}(\mathbf{z}),D_{\mathbf{z}}\phi(\mathbf{z})\im \boldsymbol{\varpi}(|\mathbf{z}|^2)\mathbf{z}\>=&
\<\sum_{\mathbf{m}\in \mathbf{R}_{\mathrm{min}}}\mathbf{z}^{\mathbf{m}}G_{\mathbf{m}}+\mathcal{R}(\mathbf{z}),D_{\mathbf{z}}\phi(\mathbf{z})\(\partial_t \mathbf{z}+\im \boldsymbol{\varpi}(|\mathbf{z}|^2)\mathbf{z}\)\>
\\&-\<\sum_{\mathbf{m}\in \mathbf{R}_{\mathrm{min}}}\mathbf{z}^{\mathbf{m}}G_{\mathbf{m}}+\mathcal{R}(\mathbf{z}),D_{\mathbf{z}}\phi(\mathbf{z})\partial_t \mathbf{z}\>.\label{eq:lFGR1}
\end{align}
The 1st term of the r.h.s.\ of \eqref{eq:lFGR1} can be written as
\begin{align}%
 &\label{eq:lFGR7} \<\sum_{\mathbf{m}\in \mathbf{R}_{\mathrm{min}}}\mathbf{z}^{\mathbf{m}}G_{\mathbf{m}},D_{\mathbf{z}}\phi(0)\(\partial_t \mathbf{z}+\im \boldsymbol{\varpi}(|\mathbf{z}|^2)\mathbf{z}\)\>+R_1(\mathbf{z}),
\end{align}
where
\begin{align*}
R_1(\mathbf{z})=&\<\sum_{\mathbf{m}\in \mathbf{R}_{\mathrm{min}}}\mathbf{z}^{\mathbf{m}}G_{\mathbf{m}},\(D_{\mathbf{z}}\phi(\mathbf{z})-D_\mathbf{z}\phi(0)\)\(\partial_t \mathbf{z}+\im \boldsymbol{\varpi}(|\mathbf{z}|^2)\mathbf{z}\)\>\\&+\<\mathcal{R}(\mathbf{z}),D_{\mathbf{z}}\phi(\mathbf{z})\(\partial_t \mathbf{z}+\im \boldsymbol{\varpi}(|\mathbf{z}|^2)\mathbf{z}\)\>,\nonumber
\end{align*}
satisfies
\begin{align}\label{eq:lFGR4}
\int_0^T |R_1(\mathbf{z}(t))|\,dt\lesssim C_0^2\epsilon_0^4.
\end{align}
Using  the stationary Refined Profile equation \eqref{eq:rfstationary}, the last line of \eqref{eq:lFGR1} can be written as
\begin{align}\label{eq:lFGR5}
-\<H\phi(\mathbf{z})+g(|\phi(\mathbf{z})|^2)\phi(\mathbf{z}),D_{\mathbf{z}}\phi(\mathbf{z})\partial_t \mathbf{z}\>
+\<D_{\mathbf{z}}\phi(\mathbf{z})(\im \boldsymbol{\varpi}(|\mathbf{z}|^2))\mathbf{z} ,\im  D_{\mathbf{z}}\phi(\mathbf{z})\partial_t \mathbf{z}\>.
\end{align}
Notice that we have
\begin{align}\label{eq:lFGR6}
\<H\phi(\mathbf{z})+g(|\phi(\mathbf{z})|^2)\phi(\mathbf{z}),D_{\mathbf{z}}\phi(\mathbf{z})\partial_t \mathbf{z}\>=\frac{d}{dt}E(\phi(\mathbf{z})),
\end{align}
and the right hand side  of \eqref{eq:lFGR2}, i.e. the left hand side of \eqref{eq:discfund},  cancels with the 2nd line  of\eqref{eq:lFGR5}.
Therefore, from \eqref{eq:discfund} with $\widetilde{\mathbf{z}}=\im \boldsymbol{\varpi}(|\mathbf{z}|^2)\mathbf{z}$, \eqref{eq:lFGR2}, \eqref{eq:lFGR1}, \eqref{eq:lFGR7}, \eqref{eq:lFGR5} and \eqref{eq:lFGR6}, we have
\begin{align}\label{eq:dtE1}
\frac{d}{dt}E(\phi(\mathbf{z}))-\sum_{\mathbf{m}\in \mathbf{R}_{\mathrm{min}}}\mathbf{m}\cdot \boldsymbol{\omega}\<\eta,\im  \mathbf{z}^{\mathbf{m}}G_{\mathbf{m}}\>=\sum_{\mathbf{m}\in \mathbf{R}_{\mathrm{min}}}\<\mathbf{z}^{\mathbf{m}}G_{\mathbf{m}},D_{\mathbf{z}}\phi(0)\(\partial_t \mathbf{z}+\im \boldsymbol{\varpi}(|\mathbf{z}|^2)\mathbf{z}\)\>+R_2(\mathbf{z},\eta),
\end{align}
where
\begin{align}\label{eq:R2}
R_2(\mathbf{z},\eta)=&R_1(\mathbf{z})+\<\im \eta, D_{\mathbf{z}}^2\phi(\mathbf{z})\(\partial_t \mathbf{z}+\im \boldsymbol{\varpi}(|\mathbf{z}|^2)\mathbf{z},\im \boldsymbol{\varpi}(|\mathbf{z}|^2)\mathbf{z}\)\>
+\<\eta,D_{\mathbf{z}}\mathcal{R}(\mathbf{z})\im \boldsymbol{\varpi}(|\mathbf{z}|^2)\mathbf{z}\>\\&
+\sum_{\mathbf{m}\in \mathbf{R}_{\mathrm{min}}}(\boldsymbol{\varpi}(|\mathbf{z}|^2)-\boldsymbol{\omega})\<\eta,\mathbf{z}^{\mathbf{m}}G_{\mathbf{m}}\>+\<F(\mathbf{z},\eta),D_{\mathbf{z}}\phi(\mathbf{z})\im \boldsymbol{\varpi}(|\mathbf{z}|^2)\mathbf{z}\>,\nonumber
\end{align}
satisfies
\begin{align}\label{eq:estR2}
\int_0^T|R_2(\mathbf{z}(t),\eta(t))|\,dt\lesssim \(C_0^2\epsilon_0^2+C_0^5\epsilon_0^5\)\epsilon_0^2.
\end{align}
By Lemma \ref{lem:zpres} and $D_{\mathbf{z}}\phi(0)\widetilde{\mathbf{z}}=\widetilde{\mathbf{z}}\cdot \boldsymbol{\phi}$, the 1st term of right hand side  of \eqref{eq:dtE1} can be written as
\begin{align}
&\sum_{\mathbf{m}\in \mathbf{R}_{\mathrm{min}}}\<\mathbf{z}^{\mathbf{m}}G_{\mathbf{m}},\boldsymbol{\phi}\cdot\(\partial_t \mathbf{z}+\im \boldsymbol{\varpi}(|\mathbf{z}|^2)\mathbf{z}\)\>
=\sum_{\mathbf{m},\mathbf{n}\in \mathbf{R}_{\mathrm{min}}}\sum_{j=1}^N\<\mathbf{z}^{\mathbf{m}}G_{\mathbf{m}}, \phi_j\(-\im \mathbf{z}^{\mathbf{n}}g_{\mathbf{n},j}+r_j(\mathbf{z},\eta)\)\>\nonumber\\&
=\sum_{\substack{\mathbf{m},\mathbf{n}\in \mathbf{R}_{\mathrm{min}}\\ \mathbf{m}\neq \mathbf{n}}}\sum_{j=1}^N\mathrm{Re}\(\im \mathbf{z}^{\mathbf{m}}\overline{\mathbf{z}^{\mathbf{n}}}\)g_{\mathbf{m},j}g_{\mathbf{n},j}
+\sum_{\mathbf{m},\mathbf{n}\in \mathbf{R}_{\mathrm{min}}}\sum_{j=1}^N\<\mathbf{z}^{\mathbf{m}}G_{\mathbf{m}}, r_j(\mathbf{z},\eta)\phi_j\>,
\end{align}
where we have set $g_{\mathbf{m},j}:=\<G_{\mathbf{m}},\phi_j\>$ and used the fact that $\<\mathbf{z}^{\mathbf{m}}G_{\mathbf{m}},-\im \mathbf{z}^{\mathbf{m}}\phi_j\>=0$ due to  $G_{\mathbf{m}}$ and $\phi_j$ being $\R$ valued.
Now, for $\mathbf{m}\neq \mathbf{n}$, we have
\begin{align*}
\partial_t(\mathbf{z}^{\mathbf{n}}\overline{\mathbf{z}^{\mathbf{m}}})=&\im (\mathbf{m}-\mathbf{n})\cdot \boldsymbol{\omega} \mathbf{z}^{\mathbf{n}}\overline{\mathbf{z}^{\mathbf{m}}}+\im (\mathbf{m}-\mathbf{n})\cdot \(\boldsymbol{\varpi}(|\mathbf{z}|^2)-\boldsymbol{\omega}\) \mathbf{z}^{\mathbf{n}}\overline{\mathbf{z}^{\mathbf{m}}}\\&+D_{\mathbf{z}}(\mathbf{z}^{\mathbf{n}})(\partial_t \mathbf{z}+\im \boldsymbol{\varpi}(|\mathbf{z}|^2\mathbf{z}))\overline{\mathbf{z}^{\mathbf{m}}}+
\mathbf{z}^{\mathbf{n}}\overline{D_{\mathbf{z}}(\mathbf{z}^{\mathbf{m}})((\partial_t \mathbf{z}+\im \boldsymbol{\varpi}(|\mathbf{z}|^2\mathbf{z})))}
\end{align*}
Thus, since
 $(\mathbf{m-n})\cdot \boldsymbol{\omega}\neq 0$ from Assumption \ref{ass:linearInd}, we have
	\begin{align}\label{znmnormal}
	\mathbf{z}^{\mathbf{n}}\overline{\mathbf{z}^{\mathbf{m}}}=\frac{1}{\im ((\mathbf{m}-\mathbf{n})\cdot \boldsymbol{\omega})}\partial_t(\mathbf{z}^{\mathbf{n}}\overline{\mathbf{z}^{\mathbf{m}}})
	+r_{\mathbf{n},\mathbf{m}}(\mathbf{z}),
	\end{align}
	where
	\begin{align*}
	r_{\mathbf{n},\mathbf{m}}(\mathbf{z})=&
	-\frac{(\mathbf{m}-\mathbf{n})\cdot \(\boldsymbol{\varpi}(|\mathbf{z}|^2)-\boldsymbol{\omega}\)}{ (\mathbf{m}-\mathbf{n})\cdot \boldsymbol{\omega}}  \mathbf{z}^{\mathbf{n}}\overline{\mathbf{z}^{\mathbf{m}}}
	\\&+\frac{\im}{ (\mathbf{m}-\mathbf{n})\cdot \boldsymbol{\omega}}\(D_{\mathbf{z}}(\mathbf{z}^{\mathbf{n}})(\partial_t \mathbf{z}+\im \boldsymbol{\varpi}(|\mathbf{z}|^2\mathbf{z}))\overline{\mathbf{z}^{\mathbf{m}}}+
		\mathbf{z}^{\mathbf{n}}\overline{D_{\mathbf{z}}(\mathbf{z}^{\mathbf{m}})((\partial_t \mathbf{z}+\im \boldsymbol{\varpi}(|\mathbf{z}|^2\mathbf{z})))}\).
		\end{align*}
Then, by the hypotheses of Proposition \ref{prop:mainbounds} we have
\begin{align}\label{est:rmn}
\int_0^T|r_{\mathbf{m},\mathbf{n}}(\mathbf{z})|\,dt\lesssim C_0^2\epsilon_0^4.
\end{align}
Thus, we have
\begin{align*}
\sum_{\substack{\mathbf{m},\mathbf{n}\in \mathbf{R}_{\mathrm{min}}\\ \mathbf{m}\neq \mathbf{n}}}\sum_{j=1}^N\mathrm{Re}\(\im \mathbf{z}^{\mathbf{m}}\overline{\mathbf{z}^{\mathbf{n}}}\)g_{\mathbf{m},j}g_{\mathbf{n},j}=
\partial_tA_1(\mathbf{z})+R_3(\mathbf{z}),
\end{align*}
where
\begin{align*}
A_1(\mathbf{z})&=
\sum_{\substack{\mathbf{m},\mathbf{n}\in \mathbf{R}_{\mathrm{min}}\\ \mathbf{m}\neq \mathbf{n}}}\sum_{j=1}^N
\frac{1}{
\(\mathbf{n}-\mathbf{m}\)\cdot \boldsymbol{\omega}
}
\mathrm{Re}(\mathbf{z}^{\mathbf{m}}\overline{\mathbf{z}^{\mathbf{n}}})
g_{\mathbf{m},j}g_{\mathbf{n},j},\ \text{and}\\
R_3(\mathbf{z})&=\sum_{\substack{\mathbf{m},\mathbf{n}\in \mathbf{R}_{\mathrm{min}}\\ \mathbf{m}\neq \mathbf{n}}}\sum_{j=1}^N\mathrm{Re}\(\im r_{\mathbf{n},\mathbf{m}}(\mathbf{z})\)g_{\mathbf{m},j}g_{\mathbf{n},j}.
\end{align*}
Thus,
\begin{align*}
\sum_{\mathbf{m}\in \mathbf{R}_{\mathrm{min}}}\<\mathbf{z}^{\mathbf{m}}G_{\mathbf{m}},\boldsymbol{\phi}\cdot\(\partial_t \mathbf{z}+\im \boldsymbol{\varpi}(|\mathbf{z}|^2)\mathbf{z}\)\>
=\partial_t A_1(\mathbf{z})+R_4(\mathbf{z},\eta),
\end{align*}
where
\begin{align*}
R_4(\mathbf{z},\eta)=R_3(\mathbf{z})+\sum_{\mathbf{m},\mathbf{n}\in \mathbf{R}_{\mathrm{min}}}\sum_{j=1}^N\<\mathbf{z}^{\mathbf{m}}G_{\mathbf{m}}, r_j(\mathbf{z},\eta)\phi_j\>.
\end{align*}
By \eqref{est:rmn} and Lemma \ref{lem:zpres}, we have
\begin{align*}
\int_0^T|R_4(\mathbf{z}(t),\eta(t))|\,dt\lesssim C_0^2\epsilon_0^4.
\end{align*}
Substituting $\eta=R[\mathbf{z}]\xi -(R[\mathbf{z}]-1)Z(\mathbf{z})-Z(\mathbf{z})$ into the 2nd term of the l.h.s.\ of \eqref{eq:dtE1}, we have
\begin{align}\label{eq:lFGR8}
&\sum_{\mathbf{m}\in \mathbf{R}_{\mathrm{min}}}\mathbf{m}\cdot\boldsymbol{\omega}\<\eta,\im\mathbf{z}^{\mathbf{m}} G_{\mathbf{m}}\>
=-\sum_{\mathbf{m}\in \mathbf{R}_{\mathrm{min}}}\mathbf{m}\cdot \boldsymbol{\omega} |\mathbf{z}^{\mathbf{m}}|^2
\< R_+(\mathbf{m}\cdot \boldsymbol{\omega})P_c G_{\mathbf{m}},\im  G_{\mathbf{m}}\>\\& -
\sum_{\substack{\mathbf{m},\mathbf{n}\in \mathbf{R}_{\mathrm{min}}\\ \mathbf{m}\neq \mathbf{n}}}\mathbf{m}\cdot \boldsymbol{\omega}\< \mathbf{z}^{\mathbf{n}}R_+(\mathbf{n}\cdot \boldsymbol{\omega})P_c G_{\mathbf{n}},\im  \mathbf{z}^{\mathbf{m}}G_{\mathbf{m}}\>
+\sum_{\mathbf{m}\in \mathbf{R}_{\mathrm{min}}}\mathbf{m}\cdot \boldsymbol{\omega}\<R[\mathbf{z}]\xi-(R[\mathbf{z}]-1)Z(\mathbf{z}),
 \im \mathbf{z}^{\mathbf{m}}G_{\mathbf{m}}\>.\nonumber
\end{align}
By \eqref{znmnormal}, the 2nd term of the r.h.s.\ of \eqref{eq:lFGR8} can be written as
\begin{align*}
-\sum_{\substack{\mathbf{m},\mathbf{n}\in \mathbf{R}_{\mathrm{min}}\\ \mathbf{m}\neq \mathbf{n}}}\mathbf{m}\cdot \boldsymbol{\omega}\< \mathbf{z}^{\mathbf{n}}R_+(\mathbf{n}\cdot \boldsymbol{\omega})P_c G_{\mathbf{n}},\im  \mathbf{z}^{\mathbf{m}}G_{\mathbf{m}}\>
=\partial_tA_2(\mathbf{z})+R_5(\mathbf{z}),
\end{align*}
where
\begin{align*}
A_2(\mathbf{z})&=-\mathrm{Re}\sum_{\substack{\mathbf{m},\mathbf{n}\in \mathbf{R}_{\mathrm{min}}\\ \mathbf{m}\neq \mathbf{n}}}
\frac{\mathbf{m}\cdot \boldsymbol{\omega}}{\im \(\mathbf{m}-\mathbf{n}\)\cdot \boldsymbol{\omega}}\mathbf{z}^{\mathbf{n}}\overline{\mathbf{z}^{\mathbf{m}}}\< R_+(\mathbf{n}\cdot \boldsymbol{\omega})P_c G_{\mathbf{n}},\im  G_{\mathbf{m}}\>,\\
R_5(\mathbf{z})&=-\sum_{\substack{\mathbf{m},\mathbf{n}\in \mathbf{R}_{\mathrm{min}}\\ \mathbf{m}\neq \mathbf{n}}}\mathbf{m}\cdot \boldsymbol{\omega}\< r_{\mathbf{n},\mathbf{m}}(\mathbf{z})R_+(\mathbf{n}\cdot \boldsymbol{\omega})P_c G_{\mathbf{n}},\im  G_{\mathbf{m}}\>,
\end{align*}
with
\begin{align*}
\int_0^T |R_5(\mathbf{z}(t))|\,dt\lesssim C_0^2\epsilon_0^4.
\end{align*}
The last term of r.h.s.\ of \eqref{eq:lFGR8} can be written as
\begin{align*}
\sum_{\mathbf{m}\in \mathbf{R}_{\mathrm{min}}}\mathbf{m}\cdot \boldsymbol{\omega}\<R[z]\xi,
 \im \mathbf{z}^{\mathbf{m}}G_{\mathbf{m}}\>+R_6(\mathbf{z}),
\end{align*}
with $R_6(\mathbf{z})$ satisfying
\begin{align*}
\int_0^T |R_6(\mathbf{z}(t))|\,dt\lesssim C_0^2\epsilon_0^4.
\end{align*}
Therefore, we have
\begin{align}\label{eq:lFGR9}
\frac{d}{dt}\(E(\phi(\mathbf{z}))-A_1(\mathbf{z})-A_2(\mathbf{z})\)
=&-\sum_{\mathbf{m}\in \mathbf{R}_{\mathrm{min}}}\mathbf{m}\cdot \boldsymbol{\omega} |\mathbf{z}^{\mathbf{m}}|^2\<R_+(\mathbf{m}\cdot \boldsymbol{\omega})P_c G_{\mathbf{m}},\im  P_cG_{\mathbf{m}}\>
\\& +\sum_{\mathbf{m}\in \mathbf{R}_{\mathrm{min}}}\mathbf{m}\cdot \boldsymbol{\omega}\<R[\mathbf{z}]\xi,\im \mathbf{z}^{\mathbf{m}}G_{\mathbf{m}}\>+R_7(\mathbf{z},\eta)\nonumber
\end{align}
where $R_7(\mathbf{z},\eta)=R_2(\mathbf{z})+R_4(\mathbf{z})+R_5+R_6$.

Now, by $R_+(\boldsymbol{\omega}\cdot \mathbf{m})
=\mathrm{P.V.}\frac{1}{H-\boldsymbol{\omega}\cdot \mathbf{m}}+\im \pi \delta(H-\boldsymbol{\omega}\cdot \mathbf{m})
$  and formula (2.5) p. 156 \cite{taylor2}      and Assumption \ref{ass:FGR}, we have
\begin{align*}
\< \im G_{\mathbf{m}},   (H-\boldsymbol{\omega}\cdot \mathbf{m} -\im 0)^{-1} G_{\mathbf{m}} \> =     \frac{1}{16 \pi \sqrt{\boldsymbol{\omega}\cdot \mathbf{m}}} \int_{|k|^2=\boldsymbol{\omega}\cdot \mathbf{m}}|\widehat{G_{\mathbf{m}}}(k)|\,dS(k)\gtrsim 1,
\end{align*}
with $\widehat{G_{\mathbf{m}}}(k)$   like in Assumption \ref{ass:FGR}.
Thus, we have
\begin{align*}
\|\mathbf{z}^{\mathbf{m}}\|_{L^2(I)}^2 \lesssim \epsilon_0^2+\delta^{-1} \|\xi\|_{L^2\Sigma^{0-}(I)}^2 + \delta \|\mathbf{z}^{\mathbf{m}}\|_{L^2(I)}^2+C_0^2\epsilon_0^4,
\end{align*}
where we have used Schwartz inequality. Taking $\delta$ so that the $\|\mathbf{z}^{\mathbf{m}}\|_{L^2(I)}^2 \lesssim \epsilon_0^2+\delta^{-1} \|\xi\|_{L^2\Sigma^{0-}(I)}^2  +C_0^2\epsilon_0^4$ and using
$\|\xi\|_{L^2\Sigma^{0-}(I)} \lesssim \epsilon_0$  by  Lemma \ref{lem:estxi},  we obtain   \eqref{FGReq}.

\end{proof}

\section*{Acknowledgments}
C. was supported by a FIRB of the University of Trieste.
M.M. was supported by the JSPS KAKENHI Grant Number 19K03579, G19KK0066A and JP17H02853.

Department of Mathematics and Geosciences,  University
of Trieste, via Valerio  12/1  Trieste, 34127  Italy.
{\it E-mail Address}: {\tt scuccagna@units.it}

Department of Mathematics and Informatics,
Graduate School of Science,
Chiba University,
Chiba 263-8522, Japan.
{\it E-mail Address}: {\tt maeda@math.s.chiba-u.ac.jp}

\end{document}